\documentclass[12pt]{article}

\usepackage{amsmath}
\usepackage{amssymb}
\usepackage{units}
\usepackage{textcomp}

\usepackage{graphicx}
\DeclareGraphicsExtensions{.pdf,.png,.jpg}
\usepackage{caption}
\usepackage{subcaption}
\usepackage{float}
\usepackage{wrapfig}
 
 \long\def\/*#1*/{}
 
 \newtheorem{claim}{Claim}
\newtheorem{theorem}{Theorem}

\newtheorem{lemma}{Lemma}

\newtheorem{remark}{Remark}

\usepackage[margin=1.25in]{geometry}

\usepackage{setspace}
\newenvironment{proof}[2]{\noindent {\nopagebreak \bf Proof #1 #2.}\par}{\normalsize\noindent}
\newenvironment{definition}{\par\noindent{\bf Definition.\enskip}\ignorespaces}{\unskip\qquad \par\medskip}

\newenvironment{assumption}{\par\noindent{\bf Assumption.\enskip}\ignorespaces}{\unskip\qquad \par\medskip}

\author{Alexey Savvateev\thanks{ Dmitri Pozarskiy University, MIPT and NES CSDSI, hibiny@mail.ru.}
\and Constantine  Sorokin\thanks{NRU Higher School of Economics, and NES CSDSI, csorokin@hse.ru.} \and Shlomo Weber\thanks{Southern Methodist University, and NES, sweber@mail.smu.edu.}}

\title{Multidimensional free-mobility
	equilibrium: Tiebout revisited}

\begin{document}

\maketitle
\begin{abstract}
The paper provides consistent mathematical framework for seminal
Tiebout free-mobility model (1956). Our setting supports continuum of
consumers with multidimensional preferences and finite number of
strategic public good providers. We accommodate the most general
assumptions: providers' production function may have variable returns
to scale, our framework is rich enough to incorporate possible
externalities, spillovers, scale economies, network effects, etc.;
consumers utility may depend on choice of other agents in almost
arbitrary way. We focus on equilibrium existence, however the
questions of efficiency and stability are not left behind. The model can also be applied in several related fields, most notably in political economy.


\end{abstract}

\pagebreak

\section{Introduction}

Rigorous and consistent modeling of local public good economies has long been a serious challenge for economists. Tiebout's (1956) early conceptual thoughts still aren't structured into a fully fledged general equilibrium framework that can rival the Arrow-Debreu model for private good economies and can be easily scaled from classroom Edgeworth box  tricks to  computational general equilibrium models of national economies. This paper aims at overcoming this issue.

It's rather straightforward how to write down a local public good economy (LPGE) model that allows free-mobility of agents, much harder is to provide its consistent solution. There are three most prominent difficulties in modeling LPGE as compared to classical private goods approach. First, on demand side, in LPGE we cannot ignore externalities that agent's choice imposes on everyone around him --- having good neighbors often turns to be crucial in residence choice; so the structure of agents' preferences has to be much richer than in the private goods economies. Second, on supply side, the public good production technology just can't be imagined as having exclusively decreasing or constant returns to scale --- the cities are never built for one resident, but for many; so productions sets are non-convex. Finally, the predictions of cooperative and non-cooperative theories don't coincide\footnote{This is the only non-obvious observation here, however, there is excessive literature clarifying this issue, see, for example, discussion in Scotchmer (2002).}
 for LPGE as they do for Arrow-Debreu model, where competitive equilibrium can be approximated with the economy's core; so even the right solution concept is unclear.   

We provide a model which easily tackles the first two problems: agents can have almost arbitrary preferences concerning the aggregate actions of others and public good production technology can have various returns to scale. We prove the \emph{existence} of a non-cooperative (Nash, or, better, \emph{Tiebout-Nash}) equilibrium, so that each agent makes his best choice given what the others are doing. This result is in sharp contrast with earlier findings of Dreze et al. (2008), which state that the LPGE core \emph{has to be empty} once we allow for multidimensional preferences. These two results combined clearly give the upper hand to the non-cooperative approach, thus resolving the third difficulty from the former paragraph. Our approach is related to Aumann's work (1964), where an assumption of a continuum of agents helps to attain much sharper results on competitive equilibrium.

We see our main contribution in separating the issues of equilibrium existence and equilibrium efficiency in LPGE. We prove existence under much broader assumptions than allow for efficiency, thus we identify the two potential causes of welfare losses: externalities and wrong number of public good providers on market. Firstly, without externalities we do prove the welfare theorem, however, with sufficiently strong externalities it is very easy to construct something like the tragedy of commons.\footnote{This problem is well documented, see \cite{Bewley} for details.} As for the second issue, we prove that the equilibrium exists for \emph{any} number of active public good providers on market, and nowadays it seems very naive to assume\footnote{For example, if the public good provision market has free entry, then why at the optimal community size the profits of public good provider should be zero?} that the equilibrium size of the community equals its socially optimal size.\footnote{
Tiebout \cite{Tiebout} actually did that in a most straightforward way by stating the following two assumptions: 
\emph{(a)} For every pattern of community services set by, say, a city manager who
follows the preferences of the older residents
of the community, there is an optimal
community size. This optimum is
defined in terms of the number of residents
for which this bundle of services
can be produced at the lowest average
cost.
\emph{(b)} Communities
below the optimum size seek to attract
new residents to lower average
costs. Those above optimum size do just
the opposite. Those at an optimum try to
keep their populations constant.
} So we have to conclude that the question of whether the free mobility improves welfare\footnote{Let's agree to call it \emph{Tiebout hypothesis}.}  and to what extent cannot be answered in a general theoretical model with realistic assumptions,\footnote{The famous critique by T. Bewley (1981) puts forth the same ideas: ``If no more is true, then Tiebout's ideas lead not to a general theory but to a possibility which requires empirical verification.'' } and indeed, this question recently attracted a lot of empirical attention, Ferreyra (2007) and Banzhaf \& Walsh (2008)  are notable examples of these efforts, they have numerous and impactful successors. Our work provides a unifying theoretical background for such models.

The issues of potential inefficiency and equilibrium multiplicity leads us to the questions of equilibrium stability. The results we obtain are  striking: under very mild assumptions every equilibrium is stable with respect to deviations of small groups of similar agents. If we allow for deviations of arbitrary (i.e. \emph{diverse}) small groups, then the stability conditions become nontrivial. This at first site technical result has important welfare implications: if the optimal community size changes (for example, due to technological progress), then the old inefficient configuration may be persistent as it takes either large or diverse group to create a snowball effect, both possibilities might involve huge coordination problems.    

Local public good interpretation is in the focus of our paper. However, the same mathematical approach can be applied to some very different setups, including spatial political competition, club formation and even clustering.

To summarize, our main message can be summarized just in one line: We've found the right framework, forget equilibrium existence problems in LPGE.

\subsection*{Literature Review}

Since the 50-s this topic was revisited many times, some recent literature survey can be found in Scotchmer (2002).  Probably the first attempts to provide a consistent mathematical framework for Tiebout choice model were done by Westhoff (1977) and Greenberg and Weber (1985). The latter used cooperative game theory approach to formalise their equilibrium concept, while the former relied on individual choice. Despite very fruitful contribution to unidimensional cooperative framework by Alesina and Spolaore (1997), in subsequent research Dreze et al. (2008) have shown that the core is always empty even in the simplest 2-dimensional setup. Other important positive contributions in non-cooperative framework include a papers by Bewley (1981) and Ellickson et al. (1999), however they do not allow for endogenous community characteristics, and thus are very limited compared to our setup.  Musatov et al. (2014) have proven noncooperative equilibrium existence for a very broad one-dimensional class of models; but their results rely on some additional monotonicity assumptions. We elaborate the same noncooperative approach and utilise similar fixed-point techniques. However, one of our main contributions is exactly the departure from unidimensional world, as it allowed to establish equilibrium existence under very mild and most natural assumptions --- at least as compared to the previous literature.

\subsection*{Model Description and Results Preview}

We consider a model with a continuum of agents and a finite number of public good providers: each agent has to chose only one provider; for example, suppose that a provider runs a neighborhood or a town, where an agent can choose to reside.\footnote{We can allow for a technical ``provider'' that stands for not making a choice and living alone in the wild.} 
Agent's preferences can be fully described by his \emph{type}, which is a vector of both discrete (gender, ethnicity, education etc.) and real-valued variables (age, redistribution preferences, etc.). There is a distribution over agent's types, it should be non-atomic in real-valued components. 

Each community has a number of endogenous \emph{characteristics} that depend continuously on the community composition; the most obvious one is community's size, others may include ethnic diversity or the median of redistribution preferences --- anything goes. To get better intuition of what community characteristics are the reader might imagine themselves considering a job offer from another city --- any aggregate measure (welfare, health-care, crime, ecology, etc) important for that decision is a community characteristic.\footnote{However, as agents compose a continuum, one cannot condition their utility on a fact that another particular person lives in that community. This is a serious limitation, but a necessary one in a continuous framework.}

The public good providers may adjust public good provision \emph{parameters} maximizing their own utility or community's welfare. The agent's utility is integrable with respect to his type and  depends continuously on characteristics and public good provision parameter of all communities --- agents might take into account what's going on in neighboring communities, thus our model allows for community spillovers, network effects, trade, etc. Note that we do not need any concavity ar monotonicity assumptions on the individual agent's utility,\footnote{Provider's utility still has to be quasi-concave with respect to parameters of his choice, as there are finite number of providers and we are looking for a Nash equilibrium.} continuity is enough. Also note that, unlike our predecessors (Ellickson et. al (1999)) we do not have to model the private good economy\footnote{If needed, one can easily amend our model using the Aumann's (1964) approach.} on top of a local public good one; with continuum of agents we can balance things out with continuous group sizes and do not need divisible private goods for that.

To obtain the equilibrium existence result we basically need one non-trivial assumption: \emph{agents should have strict preferences}.  It is  implying that for every pair of communities, every possible community characteristics and public good provision parameters the set of indifferent agents is small (i.e. has zero measure). We show that this property holds generically in our setup.

\begin{wrapfigure}{R}{0.5\textwidth}
   	\includegraphics[scale=0.6,trim =  30mm 120mm 20mm 30mm]{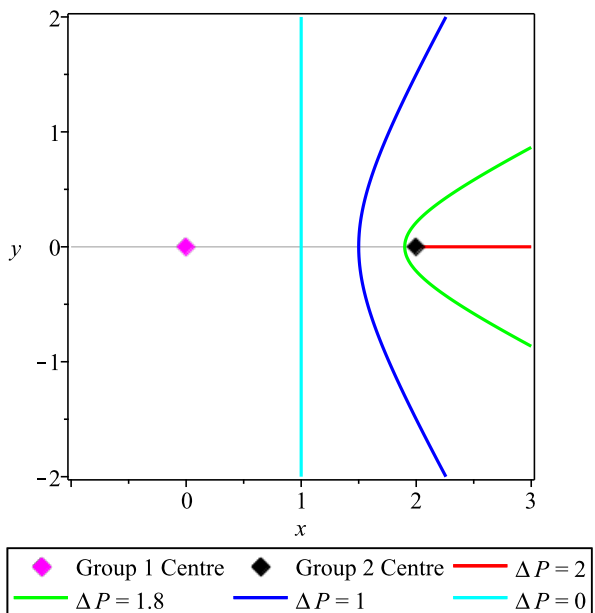} 
	\caption{Sets of indifferent agents given the difference of prices. Note the horizontal line at $\Delta P=2$.}
\end{wrapfigure}

 We will also refer to this assumption as a \emph{hyperbola property}. Its name comes out of a very simple model, with its help we will demonstrate what exactly goes wrong if this property fails. Consider agents having two-dimensional Euclidean  metric preferences: agent's cost is proportional to the distance from their ideal point and that of the community (centre).  If providers charge constant prices for joining communities and agents have quasilinear utility, then the set of indifferent agents would lie exactly on a hyperbola: the difference of prices should equal the difference of distances, and a hyperbola is exactly the locus of points where the absolute value of the difference of the distances to the two foci is a constant. Hyperbola is a small subset of a 2-dimensional space, but note that it is not the case in the  unidimensional world, where hyperbola could be a positive-measure interval. On the figure to the right it happens when difference in prices exactly equals the difference between community centres; otherwise, the set of indifferent agents is a single point or empty.

So the case of metric preferences on a real line is  a knife-edge situation where the swings of large groups of indifferent agents create discontinuities that undermine the equilibrium existence. Unfortunately, exactly this story was a starting point for too many earlier attempts to prove a general existence theorem, see Musatov et al. (2016) for details.\footnote{The strict preferences condition was used in a related field of matching theory by Azevedo and Leshno (2013). Their results allow us to extend out theory (almost ``as is'') to the case when the communities have strict caps on their sizes and Gale-Shapley algorithm is used to allocate agents to communities.}

If we go further and make one more assumption --- \emph{small group inefficiency} then the equilibrium with all the communities being non-empty also exists. Basically this condition implies that  the expenses of all the public good users go to infinity as the size of their group goes to zero.  For example this is the case when there is some small fixed cost that has to be shared among public good users. This conditions also has a clear increasing returns to scale interpretation --- however we impose no further restrictions on agent costs for larger groups. Note that this assumption perfectly fits the local public goods story --- living alone in an abandoned city is certainly a bad option.

This result is even more important than the former one because equilibrium existence is obvious once small group are inefficient --- just pile everyone into one big city, whatever the congestion costs are, they are still less than infinite costs of starting a new city. This is why the existence of equilibrium with non-empty groups is nontrivial and can't be obtained by a straightforward application of a fixed-point theorem; we rely on a refined set-intersection results to obtain the proof.

Our work parallels competitive equilibrium theory as agents and providers are ``price-takers'' --- their individual decisions do not change group sizes or any other partition characteristics. The equilibrium is also Pareto-efficient with some rather similar conditions\footnote{In the general case efficiency may fail due to externalities and interaction effects, but note that the first welfare theorem also fails if consumers care not only about their own consumption.} --- conditional on the exact number of non-empty communities. But, us we already mentioned, the number of active communities need not be socially optimal. For example, if running a community requires some fixed cost and there is no side options for the agents, then the equilibrium expenses of all agents may be arbitrary high just due to the large amount of communities and proportionally large amount of fixed costs paid.\footnote{Samuelson and Musgrave were afraid of a centralized public good under-provision due to free-riding, but in LPGE we may end up with public good over-provision due to excessive fragmentation of communities. And even private goods economy might have similar problems, as  Dixit and Stiglitz show for classical monopolistic competition model, where the free-entry condition leads to similar problems.}

To address the issue of ``equilibrium'' number of communities one needs to introduce either a model of market entry and exit process, or some form of collective decision making. First option is beyond the scope of this paper: the process of creating new communities is much more complicated than creating new firms. As for the second option --- the most natural way to introduce collective decision in our model is to allow deviations by some small (but positive measure!) groups of agents and to investigate equilibrium stability with respect to such deviations. 

It turns out that in multidimensional case under minor technical assumptions any equilibrium is weakly stable, i.e. there is no mutually profitable deviation by small $\epsilon$-ball of agents. However, the conditions for strong equilibrium stability --- no mutually profitable deviation by any small set of agents --- are much more demanding. So if equilibrium is weakly stable but not strongly  then coordinated deviation of either large group with similar preferences or small group with diverse preferences is needed to do away with the equilibrium. 

If we consider separable agent cost function including a component for distance in preference space (horizontal differentiation) and a component for a share of fixed cost (increasing returns to scale) then one important comparative statics result can be obtained. As distance costs decrease relative to fixed costs any possible nontrivial equilibria will sooner or later lose the strong stability property (if it had one). And in the limit case with zero distance costs only trivial equilibrium with just one big public good provider is stable. On the other hand, if there are no fixed costs then any equilibrium is stable. Note that since all the equilibria remain weakly stable the deviation of a group of agents with \emph{diverse} preferences is required to go from more smaller groups to fewer lager.  

This statement has important efficiency implication: suppose that at some moment public good provision was efficient, however, since then the distance costs decreased.\footnote{In geographical setting this might be due to falling transportation costs, in political setting this might be due to people becoming more tolerant.} Then to maintain efficient public good provision the number of communities should decrease, but as stated above, any equilibrium is weakly stable; strong stability failure allows to bring down an equilibrium by  a coordinated deviation of a set of agents with diverse preferences --- but it may be very hard to execute just because people with different preferences might just dislike each other. So it gives rise to another potential source of inefficiency in local public good provision.

We proceed by laying down a formal model, first its simpler version (Section 2), and next the extended case (Section 3). Next we will discuss welfare properties and equilibria stability (Section 4).

\section{Basic model}
We begin with laying down a simplified version of our model, useful to explain its core principle and derive several important comparative statics results. For now we allow only for continuous component's in agent's type, community characteristics restricted just to their sizes and no public good provision parameters. However spillovers and externalities between communities are possible. So, consider a population of economic agents distributed in $k$-dimensional vector space with a non-atomic compact-support density. We need to partition them into $n$ distinct communities such that no agent is willing to migrate. The agent's preferences are manifested in a cost function\footnote{Following our predecessors \cite{Last} we will use (wlog) agent costs instead of their utilities.} that specifies  expenses if they are willing to join any particular community. This cost function depends on community's index, agent's type and on the share of agents in each community. 

Formally we consider the following setup. Agent's preferences are fully characterized by $k$ real-valued parameters. These might be anything from his age to his redistribution preferences, and we will refer to them as agent's type. So let $x$ be a point in $\mathbb{R}^k$, $f(x)$ --- non-atomic density function with support $X$. We assume that $X$ is compact and has nonempty interior. Let $F(Y)$ be measure of set $Y\subset \mathbb{R}^k$ with respect to density $f$. Note that we neither need convexity of $X$ nor make any further assumptions about the properties of density, for example, it need not to be continuous.

Let $N=1..n$ be a finite list of all available communities. Let $c_i(x,m)$ be a function that stands for community $i$'s cost incurred to a person of type $x$ 
 should community sizes be $m$. This function combines both agent idiosyncratic preferences and communities' public good provision technology --- we do not have to treat them separately.\footnote{Once we will get to comparative statics we will separate them.}  Define $M=\{m|m_i>0,\,\sum_i m_i = 1\}$.
 We assume that $c_i$ is continuous in $m\in M$, measurable in $x$ with respect to $f$ and  bounded from above (a.e.) in $x\in X$ and $m_{-i}$ for any given $m_i\in (0,1)$, $m\in M$.

The  \emph{Tiebout-Nash equilibrium} is a partition of $X$ into $n$ communities, $I:X\rightarrow N$, such that each agent (each point in $X$) is as well off staying in his community as joining any other given the choice of all other agents. Note that it implies that before making their decisions all agents make conjectures about the actual community sizes $m$, and in equilibrium their conjectures prove to be correct. As in general equilibrium theory agents are ``price-takers,'' i.e. their individual decisions do not affect actual community sizes.

  The \emph{hyperbola property} (strict preferences) assumption:
\begin{assumption}\label{indif}
For every two distinct communities $i,j\in N$ and any community sizes $m\in M$ the set of indifferent agents has zero measure, i.e.:
$$
F(\{x|c_i(x,m)=c_j(x,m)\})=0.
$$
\end{assumption}

Despite holding generically, this hyperbola property is restrictive. For example, it is violated if $X$ is one-dimensional interval and $c$ is equal share tax plus distance to community's centre. We will consider this example in more detail later. 

Our next assumption emphasizes that the public good provision technology usually has increasing returns to scale. However, we restrict it only for small production scales, for example, once the community size is large enough we allow for congestion effects ti take over.

\begin{assumption}\label{small}
As size of a community goes to zero, the costs of its dwellers go to infinity uniformly in $x$:
 $$
 \forall A\in\mathbb{R}_+, \; \exists m_i^0\in\mathbb{R}_+:\forall m\in M:m_i<m_i^0,\;\forall x\in X \text { we have that }
 c_i(x,m)>A.
 $$ 
\end{assumption}
  In other words communities of very small size have very high per-person costs --- this is a well-known \emph{small-group ineffectiveness}. Note that this is a tricky assumption as it might  make the existence problem trivial --- indeed, if there is only one large group then no one would like to go into an empty one cause of infinite costs. So the right question is whether the nontrivial equilibrium exists. The following theorem addresses the issue.
  \begin{theorem}\label{main}
Under the assumptions above there exists an equilibrium such that all $N$ groups are non-empty.
 \end{theorem}
 
 There are two main intuitions behind this result. First, the hyperbola property guarantees continuity that earlier one-dimensional models lacked. It is so simple and elegant that we really were surprised to find out that it is actually sufficient for equilibrium existence. Even more, in two dimensional case this condition is likely to hold, as compared to the unidimensional preferences. Second, the small group inefficiency allows to balance the groups such that each on is non-empty. For example, if one group is  more beneficial to agents then the other one (other things being equal), then in equilibrium it is going to be small enough to offset this effect. 
 
 Another interpretation of the theorem is that there exists a partition such that there is no snow-ball effect when one successful group incorporates the other one in a landsliding migration. However, this equilibrium may be non-stable --- the next section investigates this issue.

 \section{Extended model}
 
 This section provides the existence result under most general assumptions. We took a lot of insight from empirical works and our idea is to incorporate everything that matters in a real world into a single model. 
 
 Now suppose that there are $q$ sets of types of agents,\footnote{Which might correspond to different ethnic groups, people of different sex and so on. Any discrete characteristic spawns a finite number of sets of types. } each has its own density $f_j(x_j)$ with compact support $X_j\subset \mathbb{R}^{k_j}$. Denote by $f$ the product density and by $X$ --- Cartesian product of supports. Let $F_j(Y_j)$ be measure of set $Y_j\subset X_j$ with respect to density $f_j$ and $F(Y)=\sum_{j=1}^q F_j(Y_j)$, $Y=Y_1\times..\times Y_r\subset X$.
 
 Let there be  $l_i$ ($L_i=1..l_i$) characteristics $\{v^i_l\}_{l\in L_i}$ of community $i$ on which agents' utility may depend should she join this community. These parameters can stand for almost anything from community's total transportation cost to the measure of community's ethnic diversity. As a special (and necessary) case of parameters above we take a share of each agent type (other than type 1) in each community. This way $m_i$ will be total size of community $i$ and the share of type 1 agents can be calculated by subtracting from $m_i$ the total share of all other types.

 Suppose that there for each community $(i\in N)$ there are $r_i$ parameters $z^i$ that public good providers are free to set the way they like. These can be public good provision level, facility locations, fees, etc.  Each provider seeks to maximise his (or community's) utility function $u_i(m,v,z)$ by choosing $z_i$. We assume that $u_i$ is continuous in all variables.

 Each agent seeks to chose community that fits his preferences best. Let $c^j_i(x,m,v,z)$ be a function that stands for community $i$'s cost incurred to a person of type $j$ located at $x$, should community sizes be $m$, communities' characteristics --- $v$, and public good provision parameters --- $z$. We assume that $c^j_i$  is measurable in $x$, continuous in all other variables ($m\in M$) and uniformly bounded from above in $x\in X$, $m_{-i}$, $v$, $z$ for any given $m_i\in (0,1)$, $m\in M$. As in the first section, we do not have to separate here public good provision technology from agent's idiosyncratic preferences.
 
 After each agent had chosen his community we obtain a partition $I:x\rightarrow N$ represented by an indicator function. With partition we can calculate community sizes $m$ and community characteristics $v(I,z)$. We assume that community characteristics depend continuously on partition with respect to symmetric difference measure  metric and that they depend continuously  on public good provision parameters as well. For example, this would be true if 
 $$
 v^i_l = \int\limits_{x\in X} H^i_l(I(x),x)f(x)dx,
 $$  
 where $H_l$ is just some (Lebesgue) integrable function in $x$.
 
 Now we need to restate assumptions appropriately:
 
 \begin{assumption}\label{indif-general}
 	For every two distinct communities $i_1,i_2\in N$, any possible  community sizes $m$, community characteristics $v$ and public good provision parameters --- $z$ the total measure of sets of indifferent agents across types has zero measure, i.e.:
 	$$
 	\sum_{j=1}^q F_j(\{x|c^j_{i_1}(x,m,v,z)=c^j_{i_2}(x,m,v,z)\})=0.
 	$$
 \end{assumption}
 The small group inefficiency assumption remains almost as is:
 
 \begin{assumption}\label{small-general}
 	As size of community goes to zero its costs go to infinity:
 	$$
 	c^j_{i}(x,m_i,m_{-i},v,z)\rightarrow +\infty \text{ as } m_i\xrightarrow[]{} 0.
 	$$ 
 \end{assumption}
 
 In addition we need an assumption on the structure of set of possible community characteristics:
 \begin{assumption}\label{structure1} 
 	The set of all possible community characteristics $V$ is homeomorphic to $l=\sum_{i=1}^n l_i$ dimensional cube.\footnote{This assumption can be relaxed using sharper topological notions, but at a cost of exposition.} From now on we will assume that this set is actually a unit cube.
 \end{assumption}
 
 And also a standard assumption about the structure of a public good providers game:
 \begin{assumption}\label{structure2}  
 	For each community,  the set of all possible public good provision parameters $Z_i$ is compact and convex. Without generality loss we will take it as $r=\sum_{i=1}^n r_i$ dimensional unit cube. The choice of provider may be restricted to some nonempty compact convex subset $\bar Z_i(m,v)\subset Z_i$ given $m$ and $v$. We will assume that  $\bar Z_i(m,v)$has closed graph in $(m,v)$. Providers' utility functions are quasi-concave with respect to parameters of their choice.  
 \end{assumption}
 
 A partition and a set of public good provision parameters is \emph{Tiebout equilibrium} if no agent is willing to change his community given community sizes, characteristics and public good provision parameters; no provider is willing to modify his parameters given his community. 
 
 The following theorem resembles our first theorem (\ref{main}), however it is much more general:
 
 \begin{theorem}\label{mainmain}
 	Under the assumptions above there exists an equilibrium such that all $N$ groups are non-empty.
 \end{theorem}

 \section{Efficiency and Stability}
 
 We begin this section by laying down a benchmark example that illustrates the key trade-off of this model in the simplest way. Next, we will return to  our model from section 2 and state some general results for it.
 
  Suppose that there is two dimensional preference space and some density with compact convex support $X$ (for example, unit circle). Any $x\in X$ is an ideal point for some consumers, their numbers are represented by a density function $f(x)$. There are $N=1..n$ public good providers, each has his own distinct ``program'' point $x^i$, which is constant. The Euclidean distance between public good provision ``program'' and consumer's ideal point is his idiosyncratic disutility from using that particular public good. The public good production technology has zero marginal cost and constant fixed cost $g$. The cost is distributed equally among everyone using that public good. So, in this case we have 
 $$
 c_i(x,m_i)=||x-x_i||+\frac{g}{m_i}.
 $$  
Note that in this case all the key assumptions are satisfied as the set of indifferent consumers is always a hyperbola --- a zero-measure curve in two-dimensional space, and small groups are inefficient because every consumer bears equal fixed cost share. The existence theorem is non-trivial because there always is an equilibrium with just one nonempty group.  
 
Stability analyses heavily relies on our ability to differentiate the cost functions, so we have to make additional assumption regarding the model primitives. However, the example above would help to keep things simple, as all the ideas below perfectly apply to it.
 
For a general model of this section, we will assume that $k=dim(X)$ is two or greater, $X$ is a convex set, function $c_i(x,m)$ does not depend on $m_{-i}$, has continuous, negative and bounded away from zero derivative in $m_i$ for all $m_i\in (0,1)$; in addition, these functions should  have continuous derivative in $x$ for all possible equilibrium indifference points.\footnote{Note that in case of distance costs in the example above center position is not a possible equilibrium indifference point --- otherwise that group would have zero size.} We will also assume that for all such points (indifference between groups $i$ and $j$)
 $$
\left|\left| \frac{\partial c_j}{\partial x}-\frac{\partial c_i}{\partial x}\right|\right| > const>0.
 $$
 For strictly convex metric distance costs and equal share participation (as in the example above) these assumptions are satisfied. 
 
 First we are going to establish an analog of the first welfare theorem. The key assumption we need is that the cost function is \emph{separable}, i.e. $c(x,m)=c^d(x)+c^s(m)$ --- exactly as the one in the example above.
  \begin{theorem}\label{Pareto}
 If agent's cost function is separable, then every equilibrium is Pareto optimal given the number of non-empty groups.
 \end{theorem}

We need to make two comments here. First, without the separability condition the result obviously fails as it's very easy to construct prisoner's dilemma-type externalities.\footnote{Bewley (1981) is a good example here.} Second, as equilibrium exists for any number of non-empty communities the welfare of agents may change almost arbitrary with the number of active public good providers; however, it's clear that if there is a lot of providers on market and public good provision requires a fixed cost then the equilibrium with fewer communities might be better for all of the agents. 

The possibility inefficient public good provision due to excessive number of communities is one of the main motivations for studying equilibrium stability.  We will proceed by defining the weak stability notion.
 
 \begin{definition}
 An equilibrium is weakly stable if there exists $\epsilon>0$ such that no subset of $\epsilon$-ball ($B_\epsilon$) of agents can move from one group to another with a mutual benefit.\footnote{For exposition reasons we only consider bilateral stability here, our notions and results can be trivially generalized to a multilateral case}
 \end{definition}
 
 Small stability is the most natural way to describe a small deviation. The key idea here that a small and closely-related group migrates from one community to the other. Surprisingly, under the assumption above every equilibrium is stable with respect to such deviations.
 
 \begin{theorem}\label{Weak}
 Every equilibrium is weakly stable.
 \end{theorem}

The intuition of this result is straightforward --- suppose that a ``diameter'' of this group is $\epsilon$. Then the worst-case per-person distance losses are of the order of $\epsilon$, while gains from scale economy are of the  $\epsilon^k$ order with $k>1$. Not surprisingly, for sufficiently small epsilon one is always greater than another. The following definition of strong stability removes the group diameter constraint, so that stability conditions become non-trivial. 
 
  \begin{definition}
 An equilibrium is strongly stable if there exists $\epsilon>0$ such that no set of agents with measure $\epsilon$ can move from one group to another to a mutual benefit.
 \end{definition}
 \begin{theorem}\label{Strong}
 If in an equilibrium for all $n-1$ dimensional measurable subsets $T_{ij}$ of group borders ($Brd_{ij}$) there exists an agent $y\in T_{ij}$ such that the following holds:
 $$
\int\limits_{T_{ij}} \frac{f(x) }{\left|\left| \frac{\partial c_j}{\partial x}-\frac{\partial c_i}{\partial x}\right|\right|} dw + \frac{1}{\frac{\partial c_j(y,m_j)}{\partial m_j}}<0.
$$
  then that equilibrium is strongly stable.
 \end{theorem}

 \begin{remark} The scale gains $\frac{\partial c_j(y,m_j)}{\partial m}$ do depend on $y$ in the general case, so we have to consider all possible deviating sets, as it may be that some agents on the border benefit from scale economy greater than the others. However, if the cost function is separable in group sizes and distances then the theorem above becomes much simpler. Because $\frac{\partial c_j(y,m_j)}{\partial m_j}$ does not depend on $y$ we should take the largest deviating set possible --- the whole border: $T_{ij}=Brd_{ij}$.
 \end{remark}
 
 Our definitions of stability were designed to illustrate the nature of deviations that may cause snowballing migration that would result in significant changes in the equilibrium structure. The results indicate that either the deviating group should be large or it has to consist of agents with \emph{diverse} preferences. One way or another, that deviation may be hard to execute, either because of coordination problem in large groups or cause of internal conflict in small but diverse ones.  
 
The conditions for strong stability may seem rather complicated, however, they still might help to estabish a key comparative statics result of the paper. Let's suppose that the distance costs are decreasing, then the strong stability would sooner or later disappear. 
 
 \begin{theorem}\label{CompStat}
 If $\frac{\partial c_j}{\partial m}$ is uniformly bounded away from zero for all equilibrium partitions, then as $\frac{\partial c_j}{\partial x}\rightarrow 0$ any nontrivial equilibrium becomes non strongly stable, however, it remains weakly stable. On the other hand,  if $|| \frac{\partial c_j}{\partial x}-\frac{\partial c_i}{\partial x}||$ is uniformly bounded away from zero in all possible equilibrium indifference points, then as $\frac{\partial c_j}{\partial m}\rightarrow 0$ any nontrivial equilibrium becomes strongly stable.
 \end{theorem}

We might consider extreme cases here: suppose that there are no distance costs, then every equilibrium is not stable as there is nothing to offset gains from the scale economy. On the other hand, if there is no scale economy effects then every equilibriums is stable, as there is nothing to gain by migrating to another group --- every agent already gets his best option. The result above establishes that situation remains the same if we are sufficiently close to these extreme situations. 

Note however, that even in the simplest example mentioned above there might be several equilibria, with some of them being stable and some not. Our numerical simulations indicate that the ``gray zone'' of parameters with mixed effects possible is rather large.

\section{Discussion}
 It's worthwhile stressing once again the key message of our research. We managed to separate to main issues with Tiebout-type models: the issue of model consistency and the issue of equilibrium efficiency. The theorem above actually closes the first issue --- the equilibrium exists under very mild and reasonable assumptions, which are more or less standard to classical economic theory. Moreover, our setup is estimation-friendly, as it incorporates into existence framework almost everything that applied researcher might want to take into account.
 
 The equilibrium efficiency is much more ambiguous as our research suggests: it might be undermined in many different ways --- from wrong number of public good providers on market to complex externalities among communities. Therefore, we would like to rephare the original Tiebout message as the ``Tiebout hypotheses:'' 
 \begin{center}
 	\emph{Free mobility improves social welfare}. 
 \end{center}
 In reality might be true or false, depending on the studied situation. And we stress that only empirical research that takes into account all the particular details can really answer that question, and that answer is necessary setting-specific.

  \section*{Appendix}
  
   \begin{proof}{of theorem}{\ref{main}}
 Take $M_e$ to be $\epsilon$-restricted ($\nicefrac{1}{n}>\epsilon>0$) simplex:
$$
M_e=\{m|m_i\geq\epsilon,\; \sum_{i=1}^n m_i=1. \}
$$
We will refer to $m\in M_e$ as \emph{nominal} group sizes. Given the nominal sizes we can form communities (partition $I:X\rightarrow N$) and measure their \emph{real} sizes:
$$
S_i(m) = \{x|x\in Arg\min\limits_{i\in N}c_i(x,m)\},\qquad f_i(m)=F(S_i(m)),\qquad (i\in N).
$$
Clearly, if $f_i=m_i$ we've found an equilibrium partition. To proceed further we need the following claims:
\begin{claim}\label{Cont}
$f(m)$ is continuous on $M_e$.
\end{claim}

\begin{proof}{of claim}{\ref{Cont}}
This result directly follows from the dominated convergence theorem. Indeed, consider a converging sequence of group sizes $\{m_l\}_{l\in\mathbb{N}}\rightarrow m$ and construct a sequence of corresponding partitions $I_l(x)$. Due to the hyperbola property assumption the sequence $I_l(x)$ of indicator functions converges almost everywhere to  the partition $I(x)$ corresponding to the limit nominal group sizes $m$. The indicator functions are trivially uniformly bounded from above, so by dominated convergence theorem we have the convergence of group sizes: 
$f_i(m_l)\rightarrow f_i(m)$. 
\end{proof}
\begin{claim}\label{Zero}
There exists $\epsilon>0$ such that 
$$
m_i\leq \epsilon \Rightarrow f_i(m)=0,\qquad (i\in N).
$$
\end{claim}
\begin{proof}{of claim}{\ref{Zero}}
First recall that due to small group inefficiency 
$ c_i(x,m_i^j,m_{-i})\rightarrow +\infty$  as  $m_i^j\xrightarrow[j\rightarrow+\infty]{} 0$.  By taking the upper bound of costs for all agents in $X$ and all the communities: we obtain the cost that should be superseded via taking sufficiently small $\epsilon$. Denote $
M^{-i}_e=\{m|m_i\geq\epsilon,\; \sum_{i\in N_{-i}} m_i=1-\epsilon \}.
$
$$
\inf_{i\in N,\;x\in X,m_{-i}\in M^{-i}_e}(c_i(x,\epsilon,m_{-i})) >  \sup_{m\in M_e,\; x\in X}\left[\inf_{j\in N}\{(c_j(x,m))\}\right].
$$
Since for every $m\in M_e$ at least one community has size above $\nicefrac{1}{n}$ all the supremums and infimums above are well-defined due to uniform boundless assumption. Thus we are done.
\end{proof}

Now the road is short --- define the following sets:
$$
A_i=\{m\in M|f_i(m)\geq m_i\},\qquad (i\in N).
$$
Observe that all these sets are closed due to claim \ref{Cont}. Since $\sum_i f_i = \sum_i m_i = 1$ (due to regularity assumption) their union covers $M_e$. Now take any $I\subset N$ and consider face of simplex $M$ spanned by vertices from $I$: any community $j$ not in $I$ has $\epsilon$ nominal mass and zero real mass on that face due to claim \ref{Zero}. Therefore $m$ on face $I$ does not belong to any set $A_j$, $j\notin I$, so face $I$ is covered by $\cup_{i\in I}A_i$. Now observe that the intersection of $A_i$, $(i\in N)$ is the equilibrium we seek. Indeed
$$
\forall i \in N: f_i\geq m_i \text { and } \sum_i f_i = \sum_i m_i = 1 \Rightarrow f_i=m_i.
$$
But such intersection exists due to the celebrated Knaster--Kuratowski--Mazurkiewicz lemma.
\begin{lemma}
Suppose that a simplex $\Delta_m$ is covered by the closed sets $C_i$ for $i \in I=\{1,\dots,m\}$ and that for all $I_k \subset I$ the face of $\Delta_m$ that is spanned by $e_i$ for $i \in I_k$ is covered by $C_i$ for $i \in I_k$ then all the $C_i$ have a common intersection point.
\end{lemma}
 \end{proof}
  
   \begin{proof}{of theorem}{\ref{Weak}}
Take an equilibrium partition and then pick an agent $x^0\in X$ such that she is indifferent between joining groups $i$ and $j$, i.e.: $c_i(x^0,m_i)=c_j(x^0,m_j)$. Due to our assumptions the sets containing such agent are the only candidates for profitable deviation; if some compact set contains only agents that strictly prefer one group over the other, then we can find such small $\epsilon$ that every $\epsilon$-ball subset of the that set would be still strictly better of staying in original group then joining any other (possible gains are proportional to $dm_i\sim\epsilon^k$).   

Note that in our setup the gains of a deviating group originate from ``scale economy'' and depend on group size $m$, while the losses are always due to ``distance costs'' and depend on location $x$. Otherwise it is not an equilibrium partition: a single deviating point does not change group sizes --- so no ``distance gains.''

 Now assume that each agent in some small group $V$ is ready to take an additional cost $d\gamma$ in attempt to 
move from group $i$ to another group $j$. We can easily find the distance $dx$ from indifferent agent $x$ that implies such cost:
$$
d\gamma=\left(\frac{\partial c_j}{\partial x}-\frac{\partial c_i}{\partial x}\right)dx
$$ 
Note that the border between groups $i$ and $j$ is given by the equation 
$$
c_i(x,m_i)=c_j(x,m_j)\Leftrightarrow c_j(x,m_j)-c_i(x,m_i)=0,
$$
so vector $\nicefrac{\partial c_j}{\partial x}-\nicefrac{\partial c_i}{\partial x}$ is normal to the border.  To find the most distant agent that pays cost no more than $d\gamma$ we take $dx$ proportional to $\nicefrac{\partial c_j}{\partial x}-\nicefrac{\partial c_i}{\partial x}$ and denote $dh=||dx||$.

Now we need to find a subset $V$ of an $\epsilon$-ball around $x_0$ such that: (a) every agent is in group $i$, (b) each agent suffers costs no more than $d\gamma$. Let $f^0$ denote the density in $x_0$ and $W_{\epsilon}$ --- $(k-1)$-dimensional ball containing $x_0$, subset of $B_{\epsilon}$. Then, by evaluating the size of $V$ from above (up to a linearisation in $dx$) we have:  
$$
dm_V=Vol(W_{\epsilon})*f^0*dh.
$$
For the deviation to be non-profitable the costs of the most hurt agent should exceed his benefits, i.e.:
$$
\left|\left|\frac{\partial c_j}{\partial x}-\frac{\partial c_i}{\partial x}\right|\right| dh>-\frac{\partial c_j}{\partial m}dm_V \Leftrightarrow \left|\left|\frac{\partial c_j}{\partial x}-\frac{\partial c_i}{\partial x}\right|\right|>-f^0\frac{\partial c_j}{\partial m}Vol(W_{\epsilon})
$$
Due to our assumptions $\||\nicefrac{\partial c_j}{\partial x}-\nicefrac{\partial c_i}{\partial x}||$ is bounded away from 0, $f^0$ is just bounded, $W_\epsilon$ can be taken arbitrary small, and the only thing that remains is to show that $\frac{\partial c_j}{\partial m}$ is bounded for all \emph{equilibrium} partitions\footnote{Due to small group inefficiency assumption it is clearly unbounded if we consider all the feasible partitions}. 

To prove it note that as we have stated in the theorem \ref{main} proof each group has a technically minimum equilibrium size and this size is non-zero (Claim 2). By our assumption the cost function is smooth, so its derivative is continuous and it is bounded on any compact subset of $M$ (recall that $M$ itself is an open simplex). 
 \end{proof}
  
   \begin{proof}{of theorem}{\ref{Strong}}
 The result can be obtained by integrating the key inequality of theorem \ref{Weak} proof. Exactly, assume that the deviating agents agree to to take an additional cost $d\gamma$. Then for each agent $x$ on the border of groups $i$ and $j$ we can convert this cost into distance from border $dh$:  
 $$
d\gamma=\left|\left|\frac{\partial c_j}{\partial x}-\frac{\partial c_i}{\partial x}\right|\right|dh
$$
 Now we can take a border subset $T_{ij}\subset Brd_{ij}$ and obtain the corresponding size of a deviating group:
 $$
dm_{T_{ij}}=d\gamma \int\limits_{T_{ij}} \cfrac{f(x)dw}{\left|\left|\frac{\partial c_j}{\partial x}-\frac{\partial c_i}{\partial x}\right|\right|}.
$$
 For the deviation to be non-profitable there should be at least one agent $y$ in a deviating group who doesn't like the idea, i.e. his distance costs are not compensated by scale gains:
 $$
 d\gamma>-\frac{\partial c_j(y,m_j)}{\partial m}dm_{T_{ij}}=-\frac{\partial c_j(y,m_j)}{\partial m}d\gamma \int\limits_{T_{ij}} \cfrac{f(x)dw}{\left|\left|\frac{\partial c_j}{\partial x}-\frac{\partial c_i}{\partial x}\right|\right|}.
 $$
 After simplification we obtain the needed expression:
  $$
\int\limits_{T_{ij}} \frac{f(x) }{\left|\left| \frac{\partial c_j}{\partial x}-\frac{\partial c_i}{\partial x}\right|\right|} dw + \frac{1}{\frac{\partial c_j(y,m_j)}{\partial m_j}}<0.
$$
Thus we are done.
 \end{proof}
  
   \begin{proof}{of theorem}{\ref{CompStat}}
 The second part of the theorem's first claim is a direct implication of theorem \ref{Weak}. The other claims obviously follow from theorem \ref{Strong}, indeed, the conditions for strong stability are:
  $$
\int\limits_{T_{ij}} \frac{f(x) }{\left|\left| \frac{\partial c_j}{\partial x}-\frac{\partial c_i}{\partial x}\right|\right|} dw + \frac{1}{\frac{\partial c_j(y,m_j)}{\partial m_j}}<0.
$$
In both cases one term goes to infinity while the other remains bounded from above. Note that due to our assumptions (existence of derivatives) the measure of any border remains finite.
 \end{proof}
  
   \begin{proof}{of theorem}{\ref{Pareto}}
 The proof is very straightforward: take one agent from each group that benefits best (i.e. has the lowest $c^d$ component value). Take any other partition --- we claim that at least one of them is worse off if group sizes change. Indeed, since the utility of an agent strictly decreases with the size of his group (small group inefficiency plus derivative is bounded away from zero) and the sum of group sizes remains constant with any other group sizes at least one of them would be worse off should their group assignment remain the same. But the latter is always the case --- if the agent who is best positioned for some group is willing to leave it then this group should be empty. We can finish the proof by pointing out that all the agents make optimal decisions given group sizes --- therefore the utility of other agents cannot be improved without changing the sizes.
 \end{proof}

 \begin{proof}{of theorem}{\ref{mainmain}}
The proof follows the lines of Theorem's \ref{main} proof. The only difference is that we have to take a more general version of KKM lemma, since we no longer have one simplex but a product of simplexes (a $n$-dimensional cube is a product of $n$ 1-dimensional simplexes). Indeed, take $\epsilon>0$, constrained group size simplex $M_e$ and some nominal sizes, characteristics and parameters $(m,v,z)\in M_e\times V\times Z$, $z\in\bar Z(m,v)$. With the help of the variables above we can construct a partition $I[m,v,z](x)$ such that every agent makes optimal choice given $(m,v,z)$. Now we calculate real values of group sizes and characteristics:
\begin{multline}\nonumber
S^j_i(m,v,z) =\{x|I[m,v,z](x)=i\}= \{x|x\in Arg\min\limits_{i\in N}c^j_i(x,m,v,z)\},\\
 f_{ij}^m(m,v,z)=F_i(S^j_i(m,v,z)), \qquad
 f_{il}^v(m,v,z)=v^i_l(I[m,v,z],z). $$
\end{multline}
Next we calculate providers' best responses to $f_{ij}^m(m,v,z)$ and $f_{il}^v(m,v,z)$:
$$
f_{i}^r(m,v,z)=Arg\max_{\bar z^i \in \bar Z_i(f^m, f^v)} u(f^m, f^v,\bar z^i,z^{-i})
$$

\begin{claim}\label{ContCont}
$f^m(m,v,z)$ and $f^v(m,v,z)$ are continuous on $M_e\times V\times Z$, $f^r(m,v,z)$ has close graph and is convex-valued.
\end{claim}

 \begin{proof}{of claim}{\ref{ContCont}}
The first part of the proof is a direct generalisation of claim's \ref{Cont} proof, the second one is standard lemma used to prove Nash equilibrium existence in game theory. 
 \end{proof}

\begin{claim}\label{ZeroZero}
There exists $\epsilon>0$ such that $\forall m_{-i},v,z$
$$
m_i\leq \epsilon \Rightarrow f^m_i(m,v,z)=0,\qquad (i\in N).
$$
\end{claim}

\begin{proof}{of claim}{\ref{ZeroZero}}
The proof directly repeats claim's \ref{Zero} proof.
\end{proof}

Now we need to modify the construction of theorem's \ref{main} proof to make use of KKM product lemma.  We will use dual covering lemma from Freund (1986) (covering lemma 2). Let $T^{\alpha}$ be a standard $\alpha$-dimensional simplex, $T=T^{\alpha_1}\times...\times T^{\alpha_{\beta}}$ be a simplitope, i.e. product of $\beta$ simplexes. Let $\Gamma= \{(j,k)|j\in\{1..\beta\},k\in\{1..\alpha_j\} \}$ be a set of indexes. We have the following lemma:
  \begin{lemma}\label{PKKM}
Let $T^{jk}$, $(j,k)\in \Gamma$, be a family of closed sets such that
$\cup_{(j,k)\in N}T^{jk} = S$, and $T^{jk}\supset \{t \in T | t^j_k= 0\}$ for each $(j,k)\in \Gamma$. Then there exists $j \in \{1, . .., \alpha \}$ such that $\cap_{k=1..m_j} T^{jk}\neq\emptyset$.
\end{lemma}

If there is just one simplex it is just Scarf's dual KKM lemma. 

In our case simplex $T$ would be a product of $M_e$, $V$ and $Z$, thus we have a product of $n-1$ dimensional simplex and $r+l$ 1 dimensional simplexes (recall that $V$ and $Z$ are cubes). 
Now we are ready to define the set that cover $T=M_e\times V \times Z$. In order to maintain standard simplicial  structure let's reindex the variables in the following way: $t=(m,v,z)$, $t_i^1=m_i$, $t_1^j=v^i_p$, $t_2^j=1-v^i_p$, $t_1^j=z^i_q$, $t_2^j=1-z^i_q$; $j=2..1+l+r$. Here $j$ changes along with $i$, $p$ and $q$ to maintain continuous numbering. Denote by $\Gamma$ the set of all such indexes. Slightly abusing notation we will use $f$ correspondence as it was defined on $T$ using appropriate indexation. We define $f_2^j(t)=1-f_1^j(t)$. Denote by $\overline{f}_k^j(t) = \min {f}_k^j(t)$ (recall that $f$ is a correspondence in some variables). Now we a ready to define the covering sets:
\begin{multline}\nonumber
A_k^{j} = \{t\in T | \overline{f}_k^j(t) -t_k^j\geq \overline{f}^p_q(t) -t^p_q,\;\forall (p,q)\in \Gamma  \}
\cup \\
\{t\in T | t_k^j=0,\text{ if } j=2..1+l+r\}\cup \{t\in T | t_k^j=\epsilon,\text{ if } j=1\}.
\end{multline}

Due to regularity assumption their union covers $M$.  Observe that all these sets are closed due to claim \ref{ContCont}.

Thus due to lemma \ref{PKKM} there exists $j \in \{1, . .., \alpha \}$ such that $\cap_{k=1..m_j} T^{jk}\neq\emptyset$ --- other assumptions are satisfied by construction. We claim that $t\in \cap_{k=1..m_j} T^{jk}$ is the equilibrium we seek. Indeed, due to continuity (closeness) and convexity we obtain that $t\in f(t)$, moreover, this point cannot lie on the boundary of first simplex (group sizes) due to lemma \ref {ZeroZero} --- the argumentation for it is the same as in theorem's \ref{main} proof.  
 \end{proof}

\end{document}